\documentclass[10pt,a4paper,twoside,final]{article}

\usepackage{amscd}
\usepackage{amsfonts}
\usepackage{amsmath}
\usepackage{amssymb}
\usepackage{amstext}
\usepackage{amsthm}
\usepackage{chngcntr}
\usepackage{color}
\usepackage[nodate]{datetime}
\usepackage{dsfont}
\usepackage{enumerate}
\usepackage{fancyhdr}
\usepackage[hmargin=30mm,top=35mm,bottom=45mm,a4paper]{geometry}
\usepackage{graphicx}
\usepackage[utf8]{inputenc}
\usepackage{mathrsfs}
\usepackage{pifont}
\usepackage{psfrag}
\usepackage[notref,notcite]{showkeys}
\usepackage[active]{srcltx}
\usepackage{stmaryrd}
\usepackage[normalem]{ulem}
\usepackage{url}
\usepackage{wasysym}

\renewcommand{\epsilon}{\varepsilon}

\newcommand{\oo}{\infty}

\newcommand{\A}{\mathcal{A}}
\newcommand{\C}{\mathcal{C}}
\newcommand{\F}{\mathcal{F}}
\newcommand{\G}{\mathcal{G}}

\newcommand{\N}{\mathcal{N}}

\renewcommand{\S}{\mathcal{S}}
\newcommand{\T}{\mathcal{T}}
\newcommand{\U}{\mathcal{U}}
\newcommand{\V}{\mathcal{V}}

\newcommand{\EE}{\mathbb{E}}
\newcommand{\NN}{\mathbb{N}}
\newcommand{\PP}{\mathbb{P}}
\newcommand{\RR}{\mathbb{R}}
\newcommand{\ZZ}{\mathbb{Z}}

\newcommand{\CC}{\mathscr{C}}

\newcommand{\LL}{\mathscr{L}}

\newcommand{\dd}{{\mathrm d}}

\newcommand{\sgn}{\operatorname{sgn}}

\newcommand{\I}{\mathds{1}}

\theoremstyle{plain}
\newtheorem{theorem}{Theorem}[section]

\newtheorem{lemma}[theorem]{Lemma}
\newtheorem{proposition}[theorem]{Proposition}
\newtheorem*{assumption*}{Assumption}
\newtheorem*{lemma*}{Lemma}
\newtheorem*{proposition*}{Proposition}
\newtheorem*{theorem*}{Theorem}
\theoremstyle{definition}
\newtheorem*{definition*}{Definition}
\theoremstyle{remark}

\newtheorem*{claim*}{Claim}

\newtheorem*{remark*}{Remark}

\counterwithin{figure}{section}
\counterwithin{equation}{section}

\newcommand{\mydate}{{\today}. \currenttime}

\date{}

\renewcommand{\baselinestretch}{1.15}
\begin{document}

\title{
Stability of the Greedy Algorithm on the Circle
}

\author{
Leonardo T. Rolla, Vladas Sidoravicius
\\ \small
Argentinian National Research Council, University of Buenos Aires
\\ \small
NYU-ECNU Institute of Mathematical Sciences at NYU Shanghai
\\ \small
Courant Institute of Mathematical Sciences, New York University
}

\maketitle

\begin{abstract}
We consider a single-server system with service stations in each point of the circle. Customers arrive after exponential times at uniformly-distributed locations. The server moves at finite speed and adopts a greedy routing mechanism. It was conjectured by Coffman and Gilbert in~1987 that the service rate exceeding the arrival rate is a sufficient condition for the system to be positive recurrent, for any value of the speed. In this paper we show that the conjecture holds true.
\end{abstract}

This preprint has the same numbering of sections, equations theorems and figures as the the
published article
``\emph{Comm. Pure Appl. Math. 70 (2017): 1961--1986.}''

\thispagestyle{empty}

\section{Introduction}
\label{sec:sec1introduction}

In this paper we study a greedy single-server system on the unit-length circle
$\RR/\ZZ$.
Customers arrive following a Poisson process with rate $\lambda$.
Each arriving customer chooses a position on $\RR/\ZZ$ uniformly at random and
waits for service.
If there are no customers in the system, the server stands still.
Otherwise, the server chooses the nearest waiting customer and travels in that 
direction at speed $v > 0$, ignoring any new arrivals.
Upon reaching the position of such customer, the server stays there until
service completion, which takes a random time $T$ that is independent
of the past configurations and has expectation $\mu^{-1}$.

The above system was introduced by Coffman and Gilbert in
1987~\cite{coffman-gilbert-87}, and since then became a paradigm example of a
routing mechanism
that depends on the system state.
This is the so-called \emph{greedy server}, due to the simple strategy of
targeting the nearest customer.

Continuous-space models provide natural approximations for systems with a large
number of service stations embedded in a spacial structure,
and their description is usually more transparent than the discrete-space
formulation, mostly because the latter often is obscured by combinatorial
aspects.
However, systems with greedy routing strategies in the continuum are
extremely sensitive to microscopic perturbations, and their rigorous study
represents a mathematical challenge.

It was conjectured in~\cite{coffman-gilbert-87} that the greedy server on the
circle should be a stable system when $\lambda<\mu$ for any $v>0$.
Since then, a number of related models have been proposed and studied.
Stability was verified under light-traffic assumptions, i.e., for
$\lambda$ and $\mu$ fixed and $v$ large enough,
and for the greedy server
on a discrete ring
$\ZZ/n\ZZ$.
However, these approximations were unable to identify and tackle the main
difficulty of this system, which is is due to the
interplay between the server's motion and
the environment of waiting customers that surround it.
This interplay is given by the
interaction resulting from the
choice of the next customer and the removal of those who have been
served.
In this paper we prove stability for the greedy server.

\begin{definition*}
We say that $t$ is a \emph{regeneration time} if the system becomes empty at
time $t$, i.e., if there is one customer at time $t-$ and no
customers at time $t+$.
Let $\tau_\clock:=\inf\{t>0:t\mbox{ is a regeneration time}\}$.
We say that the system is \emph{recurrent} if, starting from the empty state
$\clock$, there will be a.s.\ a regeneration time, i.e.,
$\PP^\clock[\tau_\clock<\infty]=1$.
We say that the system is \emph{stable}, or \emph{positively recurrent}, if
$\EE^{\clock}[\tau_\clock]<\infty$.
\label{text:recurrent}
\end{definition*}

\begin{theorem}
\label{thm:stability}
Suppose that the distribution of the service time $T$ is geometric, exponential,
or deterministic.
For any $\lambda<\mu$ and any $v>0$, the greedy server on the circle is stable.
\end{theorem}

\begin{remark*}
In our approach, it is crucial that the arrivals are Poisson in space-time.
There is a \emph{dynamic} version of the greedy server, where \emph{new arrivals
are not ignored} while the server is traveling.
This variation might be studied by similar arguments, but the dynamic
mechanism introduces some extra complications that will not be considered here.
A proof of stability for general service times having an exponential moment
follows from the same approach as presented here, requiring a little extra work
due to the lack of Markov property.
We present the proof for exponentially distributed service times with
$\mu=1$.
The cases of geometric or deterministic times only differ in notation.
\end{remark*}


For the proof of Theorem~\ref{thm:stability}, we consider a
representation for the conditional distribution of
the set of waiting customers in terms of a
stochastic evolution of profiles.
In this framework, the server learns only the information that is necessary and
sufficient to determine the next movement, and the positions of further waiting
customers remain unknown.
This approach was used in~\cite{FossRollaSidoravicius15} to show that the greedy server on the
real line is transient, which is an important ingredient in our proof of
stability.

In the remainder of this introduction, we
review some known results on the greedy server and related models,
discuss the problem of self-interaction,
describe the approach based on a stochastic evolution of profiles,
present a heuristic discussion in order to highlight
the main ideas of the proof,
and finally give a brief outline of the paper structure.

\subsection{Previous results on the greedy server and related models}
\label{sec:queueing}



Stability was verified for the greedy server on $\RR/\ZZ$ under light-traffic
assumptions~\cite{kroese-schmidt-96} and for the greedy server
on a discrete ring $\ZZ/n\ZZ$~\cite{foss-last-96,foss-last-98,meester-quant-99},
see below.
It was also shown for several related models, including a class
of non-greedy policies~\cite{kroese-schmidt-94},
a gated-greedy variant on convex spaces~\cite{altman-levy-94},
and random non-greedy servers on general spaces~\cite{altman-foss-97}.
See~\cite{rojasnandayapa-foss-kroese-11} and references therein for a recent
review.

The \emph{light-traffic} regime is given by
\[
  \lambda \left( \frac{1}{\mu} + \frac{1}{2v} \right) < 1.
\]
This regime was studied in~\cite{kroese-schmidt-96}, particularly the limit
$\lambda \to 0$ for which the first terms of Taylor expansions of some
performance measurements were computed.
A simple coupling argument works for proving stability under
light-traffic assumption.
In this case, $\frac{1}{2v}$ gives an upper bound for the travel time between two
consecutive services, since $\frac{1}{2}$ is the largest distance within
the unit-length circle.
Adding this bound to the service time allows a
comparison between the greedy server on $\RR/\ZZ$ and a stable $M/G/1$ system,
which proves that the former is stable.

This kind of argument could be pushed down to smaller values
of~$v$ than the above inequality allows by obtaining a stochastic
upper bound on the distance to the nearest customer better than~$\frac{1}{2}$.
But in any case, it may not be extended to general~$v>0$,
because the presence of the server interferes severely with the
conditional distribution of the locations of remaining customers.

On the other hand, stability under the general condition $\lambda<\mu$
is known to hold for the \emph{polling server} on $\RR/\ZZ$, i.e., the server
whose strategy is to always travel in the same direction.
In~\cite{kroese-schmidt-92} this fact was proven using a decomposition of the
set of waiting customers into a collection of Galton-Watson trees that turn out
to be subcritical for $\lambda<\mu$.
This decomposition provides a detailed description of the busy cycles
(sequence of configurations observed between two consecutive regeneration times)
and the stationary state, but if one only wants to prove stability, there
is a simple and robust argument.
Take $\epsilon<1-\frac{\lambda}{\mu}$ and $K$ such that
\[
  \frac{K}{\mu} + \frac{1}{v} < (1-\epsilon) \frac{K}{\lambda}.
\]
The above inequality implies that,
whenever the number of waiting customers $N$ is larger than $K$, the time it
will typically take to serve the first $N$
customers, including service and travel time, is less than the time it will
typically take for the next $N-\epsilon N$ arrivals, resulting in a net
decrease by~$\epsilon N$ on the number of waiting customers.

Simulations indicate that, under heavy traffic conditions, the
greedy
server dynamics resembles that of the polling server~\cite{coffman-gilbert-87}.
This suggests that a possible strategy for proving stability of the
greedy
server might be to adapt the above argument.
In this case, the first step would be to
understand its
local behavior, and 
a natural approach is to consider a system on an infinite
line.
A model on~$\ZZ$
was studied
in~\cite{kurkova-menshikov-97}, where it is shown that the server
is eventually going to move in a fixed random direction.

Some direct attempts also include the study of a greedy server model on
the finite ring $\ZZ/n\ZZ$, which was shown to be stable
in~\cite{foss-last-96,foss-last-98,meester-quant-99}.
Each of these references provides different arguments under a variety
of general assumptions.

Yet discrete models have not been able to grasp the microscopic
nature of the greedy mechanism in continuous space, neither on $\ZZ$ nor on
$\ZZ/n\ZZ$, and there are major
obstacles in extrapolating any approach based on a discrete approximation.
This difficulty is due to the \emph{self-interaction} of the server's path at
the \emph{microscopic level},
which takes place because
the server's trajectory influences the set of waiting customers and at the same
time is determined by the latter.

\subsection{Self-interacting motions}
\label{sec:self-interaction}

The main difficulty in studying greedy server systems in continuous spaces is
due to the interplay between the server's motion and the environment of waiting
customers that surround it.
This interplay is given by the interaction resulting from the greedy choice of
the next customer and the removal of those who have been served.
The server's path is \emph{self-repelling}, since the removal of already
served customers makes it less likely for the greedy server to take the next
step back into the recently visited regions.

In some well-known examples of self-repelling motions, the self-interaction
comes from an explicit prescription of the distribution of the next step in terms of
the past occupation times.
For the
excited random walks~\cite{benjamini-wilson-03},
perturbed Brownian
motions~\cite{carmona-petit-yor-98,chaumont-doney-99,davis-96,davis-99,
perman-werner-97},
and
excited Brownian motions~\cite{raimond-schapira-11},
whenever there is a drift, it is pushing the motion in a certain
fixed direction.
For the random walk avoiding its past convex
hull~\cite{angel-benjamini-virag-03,zerner-05}
and
the prudent walk~\cite{beffara-friedli-velenik-10,bousquet-melou-10},
there is a growing forbidden region containing the previous trajectory, which
strongly pushes the motion outwards.

For our greedy server, and also for the true self-avoiding
walk~\cite{toth-95,toth-99}, the true self-repelling
motion~\cite{toth-werner-98}, and the
Brownian motion with repulsion~\cite{mountford-tarres-08},
there is a mixture of information, and
``self-repulsion'' does not immediately imply ``repulsion towards~$\oo$'', since
the particle is allowed to cross its past path, receiving contradictory signals
from its left and right-hand sides.
In fact, some of the latter models are recurrent and some are transient.

It was clear since these models were introduced that they
could not be treated via standard methods and tools.
A lot remains to be understood even in dimension $d=1$,
and, despite the existence of a few disconnected techniques
that have proved useful in particular situations,
this rich field of study lacks a systematic basis.\footnote
{%
Except for the family of universality classes
given by the Schramm-Löwner~Evolutions~\cite{schramm-00},
which include $2$-dimensional loop-erased random
walk~\cite{lawler-schramm-werner-04-1} and several other
models~\cite{lawler-schramm-werner-04,smirnov-01,smirnov-10}.
}

% ...




The greedy server model has two particular features.
Unlike most of the above models, here there is no direct prescription of
how the past trajectory influences the future in terms of occupation times.
Moreover, this evolution is time-inhomogeneous in the sense that customers keep
accumulating, which yields an increasing bias towards least recently explored
regions with decreasing traveled inter-distances after each service.

\subsection{Stochastic evolution of profiles}

To address the issues mentioned in the above subsection, we consider
a representation of the customers environment which
reflects its randomness as perceived by the server.

More precisely, we only want to learn the information that is necessary and
sufficient to determine the next movement, and the positions of further waiting
customers should remain unknown.
Each time the server has to scan the system state to determine the position
of the next target, we acquire exactly two pieces of information: the
presence of a customer at that position and the absence of any other customer
at smaller distances, as illustrated in Figure~\ref{fig:potential}.

The arrivals are represented by a space-time Poisson Point Process $\nu
\subseteq (\RR/\ZZ) \times \RR$, and
in this approach one is ignoring the points of $\nu$ that have not yet
influenced the server's trajectory.
One can think of this scheme as re-sampling the set of waiting customers at each
departure time, according to the appropriate conditional distribution.
The latter is given in terms of the space-time region where the configuration
$\nu$ has not been revealed.
In this setting, the state of the system is given by the positions of the server
and the current customer, plus the profile corresponding to the boundary of
this region where $\nu$ is unknown.
The knowledge of this triplet determines the distribution of its future without
the need of any further information from the past, yielding a Markovian
evolution.
See Figure~\ref{fig:potential}.

\begin{figure}[b!]
\begin{center}
\includegraphics[width=.95\textwidth]{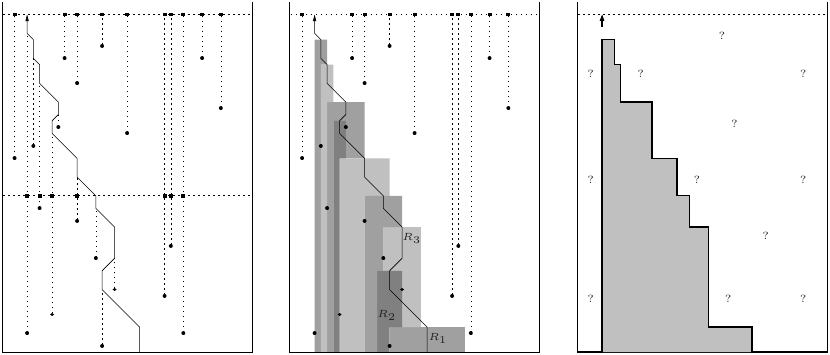}
\caption{\small
On the left, space-time evolution of the greedy server system:
Points of $\nu$ correspond to costumer arrivals and are represented by round
points.
The continuous path represents the server's motion.
For two different times we depicted the set of waiting customers with squared
points.
In the center, for the same evolution, each of the contiguous areas in different
tones of gray correspond to regions where $\nu$ is revealed at the moment when
the server needs to know the new target customer.
Each of these regions contains the point corresponding to such customer at their
lateral boundary, and no other points.
On the right, the profile corresponds to the union of all the regions where
$\nu$ has been revealed up to the time $t$ corresponding to the dotted line.
The configuration $\nu$ outside this region has been erased, since the
server's path up to time $t$ gave no information about this configuration.
The bold arrow indicates the positions of both the server and the current
customer.
}
\label{fig:potential}
\end{center}
\end{figure}

\subsection{Heuristics}
\label{sec:heuristics}

If the server is busy most of the time, the system must be stable, since on
average the service time is smaller than the inter-arrival time.
The fundamental problem in showing stability is therefore the possibility that
the server spend a long time
zigzagging on regions with low density of customers
due to a trapping configuration produced by the stochastic dynamics.

For the analogous model on the real line, this cannot be the case:
the server may zigzag for a finite period of
time, but it is bound to eventually choose a direction and head that
way~\cite{FossRollaSidoravicius15}.

On the same grounds, since the greedy routing mechanism is \emph{local}, this
can neither be the case on the circle -- at least \emph{until the server
realizes that it is not operating on the infinite line}.

Suppose we are given a configuration where the circle is crowded of waiting
customers, and, from this point on, our goal is to alleviate this situation.
We would like to say that, with high probability, after a short time
the server will choose a direction and then cope with its workload as the
polling server would.

There are two situations where the server may feel that it is on
the circle rather than on the line.
First, if it arrives at a given point $x$ for the second time after performing
a whole turn on the circle, it will encounter an environment that has been
affected by its previous visit.
This is not an actual problem, because if it happens it will imply that all the
customers which were initially present will have then been served, and typically
the server will have served more customers than new ones will have arrived.

The second difference is what poses a real issue.
The server has a tendency to go into regions that have been \emph{least recently
visited}, since in these regions the average interdistance between customers is
smaller, and they have bigger chance to attract the server via its greedy
mechanism.
This is indeed how transience is proved on~$\RR$.
Let us call the \emph{age} of a point in space the measurement in time units of
how recently it was visited by the server in the past.
On the line, the age is minimal at the server's position, and \emph{increases as
we go further away from the server}.
The new regions encountered thus become older and older, and the
server surrenders to the fact that the cleared regions it is leaving behind
cannot compete with the old regions ahead.

However, this is not true on the circle: the age profile cannot increase
indefinitely.
This gives rise to the possibility of the following tricky scenario.
Imagine that on a tiny region around some point~$x$
the system is much older than on any close neighborhood.
When the server enters this region, it will take a very long time to
finish with all
the waiting customers.
After finishing with all these customers tightly packed in space, there will no
longer be a strong difference between the ages ahead and behind the server,
who may end up going back to the region that has just been cleared,
invalidating the argument.

We deal with this difficulty by making two key observations.
First, the age of the points on the circle is monotone in some sense: there is
only one local minimum, located at the server's position, and one local
maximum~$x$, and the age increases as we move from the server towards~$x$.
Second, if the above scenario effectively happens and the server changes
direction, the new configuration may become worse in terms of the number of
waiting customers, but will be better in the sense that this sharp peak in
the age profile has
been flattened.
In order to say that the new configuration is ``better'' in this situation, we
need to quantify ``badness,'' taking into account a trade-off between
diminishing the overall workload and leveling this singular region with
excessively high concentration.
This is achieved by considering a  Lyapunov functional that combines the total
number
of customers and the maximum local density.

\subsection{Outline of the paper}

This paper is divided as follows.
In Section~\ref{sec:setup}, we give
some definitions and notation used throughout the text, and
describe the process evolution.
In Section~\ref{sec:potential}, we
introduce the stochastic evolution of profiles.
In Section~\ref{sec:lyapunov},
we define an observable $B$ that will serve as a Lyapunov
functional, along with a stopping time $\T$ so that $B_\T$ has a
downwards drift, and finally prove Theorem~\ref{thm:stability}.
The proof of downwards drift is given
in Section~\ref{sec:drift} by showing that the
greedy
server behaves most of the time like a polling server via a coupling with a
system on the infinite line, which is done in Section~\ref{sec:polling}.
The latter is studied with a block construction in Section~\ref{sec:block} and
a renewal argument in Section~\ref{sec:renewal}.


\section{Setup and notation}
\label{sec:setup}

The symbol $\preccurlyeq$ means \emph{stochastic domination} between random
elements taking value on the same partially ordered space.
Define $a \wedge b = \min\{a,b\}$, $a \vee b = \max\{a,b\}$, and $[a]^+ =
a\vee 0$.
The \emph{indicator} that $x\in J$ is denoted by $\I_J(x)$, and the
indicator that the system
is in a given state at time $s$ is denoted by $\I_{\mbox{state}}(s)$.
The complement of a set $J$ is denoted by $J^c$ when the space
where we take the complement is clear.

We consider the circle as the quotient space $\RR/\ZZ$, and for $x,y\in\RR$ we write $x\cong y$ if
$x-y\in\ZZ$.
Moreover, we identify classes of $\RR/\ZZ$ with their representants on
$\RR$, and refer to the points or their representants without distinction
unless
mentioned otherwise.
We denote \emph{arcs} on the circle by $[x,y)\subseteq \RR/\ZZ$, given by
the projection of $[x,y)\subseteq \RR$ for any pair of representants
$x,y\in\RR$ such that $x\leqslant y < x+1$.
Analogously for open or closed arcs.
We define the \emph{clockwise distance ${\vec d}$} by ${\vec d}(x,y) =
y-x\in[0,1)$.
In particular, $(x,y]=(y,x]^c$ and ${\vec d}(x,y) = 1- {\vec d}(y,x)$.
The distance on $\RR/\ZZ$ is given by $d(x,y)={\vec d}(x,y)\wedge{\vec d}(y,x)$.
We say that $f$ is \emph{increasing} on $[x,y]\subseteq \RR/\ZZ$ if
$f$ is increasing on any lifting $[x,y]\subseteq \RR$ with $x\leqslant y < x+1$;
analogously for $f$ nonincreasing, nondecreasing, etc.

\paragraph{Evolution of the greedy server system}

The state of the system at time $t$ is described by the triplet
$(\mathscr{C}_t,\S_t,\C_t)$.
Here $\mathscr{C}_t$ denotes the set of customers
present at the system, $\S_t$ denotes the position of the server, $\C_t \in
\mathscr{C}_t$ denotes the position of the customer being served or targeted by
the server, and $\C_t = \mathscr{C}_t = \clock$ when the system is empty.
The process $(\mathscr{C}_t,\S_t,\C_t)_{t \geqslant 0}$ is a strong Markov
process, whose stochastic evolution we describe now.

At all times, $t\mapsto\S_t$ is continuous and
\begin{equation}
\label{eq:movement}
\frac{\dd \S_t}{\dd t}=\V_t
:=
\begin{cases}
0, & \S_t=\C_t \mbox{ or } \C_t = \clock,
\\
v, & \S_t \ne \C_t,\ \vec{d}(\S_t,\C_t) < \vec{d}(\C_t,\S_t),
\\
-v, & \S_t \ne \C_t,\ \vec{d}(\S_t,\C_t) \geqslant \vec{d}(\C_t,\S_t),
\end{cases}
\end{equation}
in the sense of right derivative.

There are three different regimes: \emph{moving} when $\S_t\ne\C_t \in \RR/\ZZ$,
\emph{serving} when $\S_t=\C_t$, and \emph{idle} when $\C_t = \clock$.
While the system is \emph{idle}, $\S$ and $\C$ remain unchanged
until an arrival happens.
While the server is \emph{moving}, the evolution of $\S$
obeys~(\ref{eq:movement}), and $\C$ remains constant.
This regime lasts until \emph{service starts}, i.e., until the time $s$
when $\S_s=\C_s$.
During \emph{service}, the evolution of $\S$ is again given
by~(\ref{eq:movement}), $\C$ also remains constant, and
\emph{service finishes} according to an exponential clock of rate $1$.

The moments when \emph{service finishes} will be called \emph{departure times}.
At departure times $t$, the new regime will be either \emph{moving} or
\emph{idle}.
First, the current customer is removed from the system: $\mathscr{C}_{t} =
\mathscr{C}_{t-} \setminus \{\C_{t-}\}$.
Then $\C_t$ is chosen as the nearest waiting customer, if any:
\begin{equation}
\label{eq:newregime}
\C_t =
\arg \min \{d(\S_t,x):x\in\mathscr{C}_t\}
,
\quad
\mbox{or}
\quad
\C_t=\clock
\mbox{ if }
\mathscr{C}_t=\clock.
\end{equation}

During the whole evolution, \emph{arrivals} happen at rate $\lambda$.
An \emph{arrival} consists of adding to $\mathscr{C}$ a new point $z$ chosen
uniformly at random on $\RR/\ZZ$, i.e.,
$\mathscr{C}_{t}=\mathscr{C}_{t-}\cup\{z\}$.
If $\C_{t-}\ne\clock$, i.e., the server was moving or serving, this is the only
change.
If $\C_{t-}=\clock$, i.e., the system was idle, then also $\C$ is
updated by $\C_t=z$ and a.s.\ the new regime is {moving}.

\section{The process viewed from the server}
\label{sec:potential}

The process $(\mathscr{C}_t,\S_t,\C_t)_{t \geqslant 0}$ may be constructed from
two point processes:
a Poisson Point Process $\nu \subseteq (\RR/\ZZ) \times \RR_+$ with intensity
$\lambda \cdot \dd x \dd t$, each point corresponding to the arrival of a new
customer at position $x$ at time $t$,
and the Poisson Point Process
$\mathscr{T}\subseteq \RR_+$
corresponding to possible departure times (each mark $t\in \mathscr{T}$
effectively corresponds to a departure time if a customer was being served up
to time $t-$ and is ignored if the server was idle or moving).
For $u$ and $w$ denoting functions on $\RR/\ZZ$ or constants, let
\[
\Gamma_u^w = \{(x,s): x\in\RR/\ZZ, u(x) < s \leqslant w(x)\}
\subseteq (\RR/\ZZ) \times \RR_+
.
\]
The $\sigma$-algebra $\F_t = \sigma( \nu_t, \mathscr{T}_t )$, where $\nu_t =
\nu \cap \Gamma_0^t$ and $\mathscr{T}_t = \mathscr{T} \cap [0,t]$, contains all
the information about arrivals and departures up to time $t$, and consequently
about $(\mathscr{C}_s,\S_s,\C_s)_{s\in[0,t]}$.

The process $(\S_t,\C_t)_{t \geqslant 0}$ is not Markovian.
Indeed, the conditional distribution of $(\S_s,\C_s)_{s \geqslant t}$ given $\F_t$
depends on
both $(\S_t,\C_t)$ and $\mathscr{C}_t$.
Yet the only interaction between $(\S,\C)$ and $\mathscr{C}$ is
given by~(\ref{eq:newregime}).
Namely, at each \emph{departure time} $t$, $\mathscr{C}_t$ is
\emph{queried} about the \emph{nearest} waiting customer $\C_t$, if any.
The position of $\C_t$ reveals that $\mathscr{C}_t \cap
[\S_t-z,\S_t+z] = \left\{ \C_t \right\}$, where $z=d(\S_t,\C_t)<\frac{1}{2}$,
and on the other hand it gives no information about the complementary set
$\mathscr{C}_t \cap [\S_t-z,\S_t+z]^c$ of waiting customers.

In the sequel we discuss the conditional distribution of $\mathscr{C}_t$
given $(\S_s,\C_s)_{s \in [0,t]}$, the role played by this conditional law, and
the evolution of this law itself.

\paragraph{Markovianity without $\CC_t$}
By the Markov property of $(\mathscr{C}_t,\S_t,\C_t)_{t \geqslant 0}$ with
respect to $\left\{ \F_t \right\}_{t\geqslant 0}$ we have that the conditional
law of $(\S_s,\C_s)_{s\geqslant t}$ satisfies
\begin{equation}
\nonumber
\LL \left[ (\S_s,\C_s)_{s\geqslant t} \big| \F_t \right]
=
\LL \left[ (\S_s,\C_s)_{s\geqslant t} \big| (\mathscr{C}_t,\S_t,\C_t) \right]
.
\end{equation}
Let $$\G_t=\sigma\left( (\S_s,\C_s)_{s\in [0,t]} \right)\subseteq \F_t.$$
In the sequel we consider the triple $\left(\LL \big. ( \mathscr{C}_t|\G_t
),\S_t,\C_t \right)$ and study its evolution.

By the observations in the previous paragraph, the evolution $(\S_s,\C_s)_{s \in
[0,t]}$ gives information about $\nu\cap (\RR/\ZZ)\times(0,t]$ in a very precise
way.
At each departure time $s$, the new $\C_s$ is chosen
as the point of $\CC_s$ that is closest to $\S_s$.
At these times, $\C_s$ is given by $\S_s\pm z$, where $z$ is the smallest
distance for which there is a point $(\S_s\pm z,s')$ with $s'\in(0,s]$ in
$\nu$, not considering the points that correspond to customers who have already
left the system.
This reveals a rectangle $[\S_s-z,\S_s+z]\times(0,s]$ where $\nu$
has no more points that will participate in the construction of $(\CC_r)_{r>s}$,
and the law of $\CC_s$ outside $[\S_s-z,\S_s+z]$ is not affected.
For times $r$ between $s$ and the next departure time,
$\CC_r$ is given by the union of $\CC_s$ and the Poisson arrivals corresponding
to $\nu\cap (\RR/\ZZ)\times(s,r]$.
For the times $s$ when the system is in the idle state, the revealed rectangle
is the whole $\RR/\ZZ \times(0,s]$.

Iterating this argument, by time $t$ the configuration $\nu$ has been revealed
on the region given by the union of such rectangles.
Since all these rectangles have their base on $t=0$, their union is of the form
$\Gamma_0^{{w}_t}$, where $w_t(x)$ denotes the maximal height among all the
rectangles whose base contains the point $x$.
In other words, the value of $w_t(x)$ is the most recent among:
the departure times $s\in (0,t]$ such that $x\in[\S_s-z,\S_s+z]$;
and the times $s\in (0,t]$ when the system was idle.
The set of waiting customers $\CC_t\setminus \left\{ \C_t \right\}$ is thus
determined by the configuration $\nu$ on the complementary region
$\Gamma_{{w}_t}^t$.
Therefore, the conditional distribution of $\CC_t\setminus \left\{ \C_t
\right\}$ given $\G_t$ is that of an inhomogeneous Poisson process on $\RR/\ZZ$,
with local intensity at each point $x$ given by
\[
\lambda \left[ \big. t-{{w}}_t(x) \right] \dd x.
\]
In summary,
\[
\LL \left[ (\S_s,\C_s)_{s\geqslant t} \Big| \G_t \right]
=
\LL \left[ (\S_s,\C_s)_{s\geqslant t} \Big| \left({w}_t,\S_t,\C_t \right)
\right]
.
\]

Since the evolving region $\left( \Gamma_0^{{w}_t} \right)_{t\geqslant
0}$ is increased at departure times $t$ by adding a rectangle to
$\Gamma_0^{{w}_{t-}}$, this rectangle being in turn determined by $\S_t$
and $\C_t$, we have
\[
\LL \left[ \left({w}_s,\S_s,\C_s \right)_{s\geqslant t} \Big|
\left({w}_s,\S_s,\C_s \right)_{s \in [0,t]}
\right]
=
\LL \left[ \left({w}_s,\S_s,\C_s \right)_{s\geqslant t} \Big|
\left({w}_t,\S_t,\C_t \right)
\right];
\]
i.e., $({w}_t,\S_t,\C_t)$ is a Markov process with respect to its
natural filtration.
In our framework, we shall consider \[u_t = {w}_t - t \leqslant 0\] instead
of ${{w}}$, so that $(u_t,\S_t,\C_t)_{t \geqslant 0}$ is a
time-homogeneous strong Markov process.

\paragraph{Evolution of $(u_t,\S_t,\C_t)$}

The law of the evolution $(u_t,\S_t,\C_t)_{t \geqslant 0}$ is given as follows.
As before, the system may be in one of three regimes, determined by
$(\S_t,\C_t)$.

While \emph{moving} or \emph{serving}, the evolution of $\S$ and $\C$ are given
by the same rules as in the previous section: $\C$ remains constant, $\S$
satisfies~(\ref{eq:movement}), and in the serving regime \emph{service
finishes} at rate $1$.
We no longer have $\mathscr{C}$ to account for the whole set of waiting
customers.
Instead of randomly adding new customers at rate $\lambda$,
this information is now encoded in the function $u(x)$, with the rule
\begin{equation}
\label{eq:potential}
\frac{\dd u_t(x)}{\dd t}=-1 \quad \forall\ x\in\RR/\ZZ
,
\end{equation}
which rather accounts for the time period when new customers have been arriving
to the system at location $\dd x$.

At \emph{departure times},
instead of choosing the nearest point in $\mathscr{C}_t$ as
in~(\ref{eq:newregime}), we take what would be nearest point in a realization of
a Poisson Point Process on $\RR/\ZZ$ with intensity $-u_t(x)\dd x$.
More precisely, at the departure times
the system goes through an instantaneous
random transition, which may lead to either a \emph{moving} or an \emph{idle}
state, as we describe below.
Let $0<E<\oo$ and $0<U<1$ denote exponential and uniform random variables,
independent of each other and of the construction up to time $t-$.
The meaning of $E$ is that the measure of the interval that needs to be explored
before finding a point is exponentially distributed, and $U$ is important in
deciding the position of such point in the boundary of this explored interval.
The total intensity of waiting customers potentially present in the system is
given by
\[
A(u) = \int_{\RR/\ZZ} -\lambda u(x)\dd x.
\]
If $ E \geqslant A(u_{t-}) $,
take $$\C_t=\clock,$$ and the system becomes
\emph{idle}.
Otherwise, let $0<z<\frac{1}{2}$ be the unique number such that
$\int_{\S_t-z}^{\S_t+z}(-\lambda u_{t-})\dd x = E$,
let $a=-u_{t-}(\S_t-z)$, $b=-u_{t-}(\S_t+z)$,
choose
\begin{equation}
\label{eq:newcustomer}
\C_t =
\begin{cases}
\S_t - z ,& U \in (0 , \frac{a}{a+b}),
\\
\S_t + z ,& U \in [\frac{a}{a+b},1),
\end{cases}
\end{equation}
and the new regime is
\emph{moving}.
Finally, take
\begin{equation}
\label{eq:explored}
u_t(x) =
\begin{cases}
u_{t-}(x)\cdot\I_{[\S_t-z,\S_t+z]^c}(x), & E<A(u_{t-}),
\\
0,& E \geqslant A(u_{t-}).
\end{cases}
\end{equation}

The \emph{idle} regime $\C=\clock$ can only be achieved together with $u\equiv
0$ on
$\RR/\ZZ$.
While the system is idle, the state $(u,\S,\C)$ remains unchanged until the
first
customer arrival, which happens according to an exponential clock of
rate $\lambda$.
The \emph{arrival} consists of letting $\C_t=z$, where $z$ is chosen uniformly
on $\RR/\ZZ$.
Immediately after an arrival, a.s.\ the new regime is \emph{moving}.

\paragraph{Framework}

A piecewise continuous, upper semi-continuous function $u(x)\leqslant0$ on
$\RR/\ZZ$ is called a \emph{potential}.
Note that the evolution described above can start from any given potential $u$
and points $\S,\C$ such that $\C\ne\clock$ if $u\not\equiv 0$.
For shortness, the triplet $(u,\S,\C)$ will be denoted by $\U$.
Let $\PP^{\,\U}$ denote the law of $(\U_t)_{t \geqslant 0}$
starting from $\U$ at $t=0$.

We say that $u$ is a \emph{proper potential} if there exist
$x_{\min},x_{\max}\in\RR/\ZZ$ such that $u$ is nondecreasing on
the arc $[x_{\min},x_{\max}]$ and nonincreasing on the
arc $[x_{\max},x_{\min}]$, or if $u$ is monotone on any arc not containing
$x_{\max}$; see Figure~\ref{fig:properpotential}.
\begin{figure}[b!]
\begin{center}
\includegraphics[width=.95\textwidth]{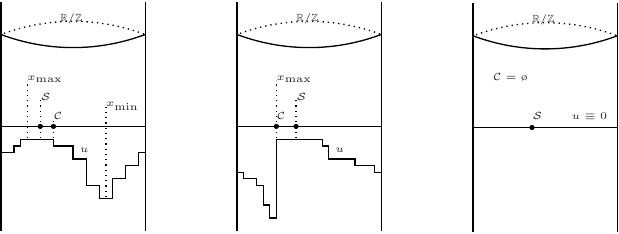}
\caption{Three examples of \emph{proper states}.}
\label{fig:properpotential}
\end{center}
\end{figure}
Given a proper potential $u$, we say that $\U=(u,\S,\C)$ is a \emph{proper state} if either $u(\C)=u(x_{\max})$, or $\C=\clock$ and $u\equiv 0$.

\begin{proposition}
\label{prop:proper}
Starting from a proper state $\U$, $\PP^{\,\U}$-a.s.\ the process
$(\U_t)_{t \geqslant 0}$ remains in proper states for every $t>0$.
\end{proposition}
\begin{proof}
When the system is idle, $\U$ is a proper state by definition.
At departure times, according to~(\ref{eq:newcustomer})-(\ref{eq:explored}) the position of $\C$ changes and $u$ is updated by increasing its value to $0$ on an arc containing both the new and the old~$\C$.
This transformation preserves the condition of~$\U$ being a proper state.
When the system state is either moving or serving, $u$ evolves according to~(\ref{eq:potential}), that is, the subtraction of a constant, which also preserves the proper state condition.
\end{proof}

\begin{remark*}
Although $\mathscr{C}_t$ is not determined by
$(\S_t,\C_t)$, we have that $\mathscr{C}_t=\clock$ if and only if $\C_t=\clock$,
and it is thus sufficient to consider the process $(\U_t)_{t \geqslant
0}$
in the study of positive recurrence, defined on page~\pageref{text:recurrent}.
This is the approach used henceforth.
\end{remark*}

\section{Proof of stability}

The goal of this this section is to prove Theorem~\ref{thm:stability}.
In Section~\ref{sec:lyapunov}, we define a Lyapunov functional~$B$ and a stopping
time $\T$.
We then state Proposition~\ref{prop:drift} about the downwards drift of $B$ at
time $\T$ and use it to prove Theorem~\ref{thm:stability}.
In Section~\ref{sec:drift}, we prove Proposition~\ref{prop:drift} making use of
Proposition~\ref{prop:polling}, which states that the total time that the server
spends traveling before time $\T$ is stochastically bounded.
In Section~\ref{sec:polling}, we prove Proposition~\ref{prop:polling} by
coupling the system on the circle with another one on the real line.
In Section~\ref{sec:block}, we show that the latter has positive probability
of being transient (and in fact ballistic) using a block construction.
In Section~\ref{sec:renewal}, we conclude with a renewal argument relying on the
uniform estimates obtained in the block construction.

We spell some formulae for later reference.
\begin{equation}
\label{eq:multidef}
\eta = 1-\lambda
,\qquad
\Psi = 2\eta^{-1}
,\qquad
\epsilon = \frac{\eta\lambda}{8}
,\qquad
\delta = \frac{\epsilon}{2 \Psi}
.
\end{equation}
The reason for these definitions will become clear as they are used in the
proof.

\subsection{Lyapunov functional and stopping times}
\label{sec:lyapunov}

Given a proper potential $u$, let
\begin{align*}
N &= N(u) = \sup_{x\in\RR/\ZZ} -u(x),
\\
B &= B(u) = A(u) + 4\epsilon N(u).
\end{align*}

Notice that the evolution of $u$ is given by~(\ref{eq:potential}) when the
state is moving or serving,
at departure times it jumps upwards according
to~(\ref{eq:explored}),
and it remains constant when the state is idle.
It thus follows that
\begin{equation}
\label{eq:uslow}
u_{t+s} \geqslant u_t - s
\qquad
\forall\ s,t\geqslant 0.
\end{equation}
Since $\lambda + 4\epsilon<1$, it follows from~(\ref{eq:uslow}) that
\begin{equation}
\label{eq:Bslow}
B(u_{t+s}) \leqslant B(u_t)+s
\qquad
\forall\ s,t\geqslant 0.
\end{equation}

Let $B_*$ denote a finite number that will be fixed later.
We claim that, for any proper state $\U$ with $B(u) \leqslant B_*$,
\begin{equation}
\label{eq:easywhenBsmall}
\PP^{\,\U}\left[ \Big. \tau_\clock < \frac{1}{2v} + 1 \right]
\geqslant \left(1-e^{-1}\right) \exp{\Big( -B_*-1-\frac{1}{2v} \Big)}
> 0
.
\end{equation}
To see why the claim is true, consider the event that the server travels towards
the nearest customer $\C$, then finishes service within $T<1$ time unit, and at
this departure time the next state given
by~(\ref{eq:newcustomer})~and~(\ref{eq:explored}) is idle.
When these events hold, since the distance $d=d(\S,\C)$ is at most
$\frac{1}{2}$, this departure time happens at
$t'=\frac{d}{v}+T < \frac{1}{2v} + 1$, implying that
$\tau_\clock < \frac{1}{2v} + 1$.
The first term on the right-hand side corresponds to the probability that $T<1$.
The second term is a lower bound for the the conditional probability that
$\C_{t'}=\clock$ given $t'$, since the latter is given by
$e^{-A(u_{t'})} \geqslant e^{-B(u_{t'})}$ which
by~(\ref{eq:Bslow}) is bounded by $e^{-B(u)-t'}$, proving the claim.

By~(\ref{eq:Bslow}) and~(\ref{eq:easywhenBsmall}), the
proof of Theorem~\ref{thm:stability}
reduces to showing that
\begin{equation}
\label{eq:finiteHittingTime}
\sup \left\{ \Big. \EE^{\U}\left[\tau_{\{B \leqslant B_*\}}\right] :
{\U
\mbox{ proper state, }B(u)<
\bar{B}} \right\} < \infty
\quad
\forall \bar{B}<\oo,
\end{equation}
where $\tau_{\{B \leqslant B_*\}}=\inf\{t:B(u_t)\leqslant B_*\}$.

Let $\U$ be a  proper state such that $B(u) > B_*$.
In the proof of~(\ref{eq:finiteHittingTime}), we study the behavior of
$B(u_t)$ at a particular stopping time $\T$ that is defined below.

Define the sets
\begin{align}
U &= \left\{x \in\RR/\ZZ: u(x) < - \textstyle \frac{N}{2} \right\}
\label{eq:valley}
, \\
G_t &= \left\{ \big. x \in\RR/\ZZ: u_t(x) > -t \right\}
.
\nonumber
\end{align}
Since $u$ is a proper potential, $U$ must be either $\RR/\ZZ$
or an open arc.
Notice that $G_0=\emptyset$ and by~(\ref{eq:uslow}) we have that $G_t$ is
nondecreasing in $t$.
By~(\ref{eq:potential})~and~(\ref{eq:explored}), it may only increase at
departure times $t$, by adding a closed arc containing $\S_t$ and $\C_t$.
Thus $G_t$ is always either $\emptyset$, or all $\RR/\ZZ$, or a closed arc
containing $\S_t$.

We define the following stopping times:
\begin{align}
\T^+ &= \T^+(u) = \Psi B(u),
\nonumber
\\
\T_\circ &= \T_\circ(u) = \inf\{t: G_t \supseteq \RR/\ZZ \},
\nonumber
\\
\T_\curlyvee &= \T_\curlyvee(u) = \inf\{t: G_t \supseteq U \},
\nonumber
\\
\T &= \T(u) = \T_\circ(u) \wedge \T_\curlyvee(u) \wedge \T^+(u)
.
\label{eq:stopping}
\end{align}

It follows from~(\ref{eq:Bslow}) that
\begin{align}
\label{eq:Bbounded}
B(u_{\T}) & \leqslant B(u) + \T \leqslant B(u) + \T^+ = (\Psi + 1) B(u).
\end{align}

A few comments are in order.
Normally, $\T$ is attained because the condition for $\T_\curlyvee$ is attained.
The deterministic time $\T^+$ is a \emph{safety caution}: it bounds the possible
damage that is caused when this condition is not attained in due time.
The presence of $\T_\circ$ in the definition of $\T$ is innocuous from a formal
point of view, since $\T_\circ \geqslant \T_\curlyvee$.
We write it to indicate that $\T_\curlyvee$ may be attained in two conceptually
different situations: either because $U$ is ``crossed'' by
$(G_s)_{s \geqslant 0}$, or
because $U$ is partly taken by $(G_s)_{s \geqslant 0}$ from one direction and
then
from the other, in which case the whole circle $\RR/\ZZ$ is taken.
See Figure~\ref{fig:stopping}.
\begin{figure}[b!]
\begin{center}
\includegraphics[width=.95\textwidth]{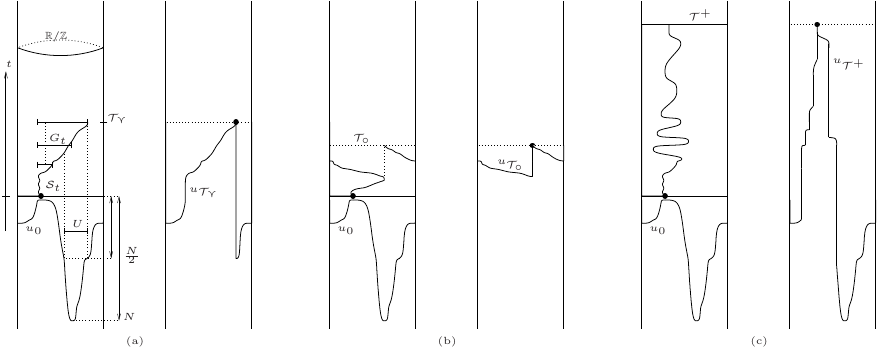}
\caption{Evolution of $G_t$ starting from a given potential $u_0$, with the new
potential $u_{\T}$ at the stopping time $\T$.
From left to right there are three pairs of graphs, each one embedded
in the space-time $(\RR/\ZZ) \times \RR$.
In~(a) we depict a typical example when $\T=\T_\curlyvee$.
In~(b) we show an instance where $\T=\T_\circ$.
Finally, in~(c) there is an example where the server remains confined for a long
time, preventing the condition for $\T_\curlyvee$ to be attained up to time
$\T=\T^+$ .
Each pair represents the system evolution and the resulting potential:
the graph on the left shows the 
parametrized curves $(\S_t,t)_{t\in[0,\T]}$ and $\left(
x,u_0(x) \right)_{x\in\RR/\ZZ}$, and on the right there is
$\left( x,u_\T(x)+\T \right)_{x\in\RR/\ZZ}$ together with the point
$(\S_\T,\T)$.
}
\label{fig:stopping}
\end{center}
\end{figure}

\begin{proposition}[Downwards drift]
\label{prop:drift}
For any proper state $\U$,
\begin{equation}
\label{eq:drift}
\PP^{\,\U} \left( \big. B(u_{\T}) \leqslant (1-\epsilon) B(u) \right)
\geqslant
1-\rho,
\end{equation}
where $\rho=\rho(B(u))$ satisfies $\rho(B)\to 0$ as $B\to\infty$,
and $\epsilon$ is defined in~(\ref{eq:multidef}).
\end{proposition}

Writing $D(\cdot) = \log \frac{B(\cdot)}{B_*}$,
(\ref{eq:Bbounded})~and~(\ref{eq:drift})
imply that
\begin{equation}
\label{eq:BoundedDriftD}
\PP^{\,\U}\left( \big. D(u_\T) \leqslant D(u) - \epsilon \right) \geqslant
1-\rho
, \
D(u_\T) \leqslant D(u) + \Psi
, \
\T \leqslant \T^+ = \Psi B_* e^D
.
\end{equation}

We are going to use the following fact, whose proof is omitted.
\begin{lemma}
\label{lemmaHittingTimeTail}
Let $(Y_n)_{n\in\NN}$ be i.i.d.\ Bernoulli random variables with
$$\PP(Y_1=\Psi) = 1-\PP(Y_1=-\epsilon) = \rho.$$
Write $\PP^s$ for the law of $(S_n)_{n\in\NN}$ given by $S_n = s + Y_1 + \cdots
+ Y_n$, and define $\sigma = \inf\{n:S_n \leqslant 0\}$.
Then there exists $\rho_*>0$ such that
$\EE^s[\sigma e^{\Psi\sigma}]<\infty$
for any $\rho \leqslant \rho_*$ and $s<\infty$.
\end{lemma}

\begin{proof}[Proof of Theorem~\ref{thm:stability}]
We need to show~(\ref{eq:finiteHittingTime}).
First we use Proposition~\ref{prop:drift} to fix the value of $B_*$ with the
property that $\rho(B) \leqslant \rho_*$ for any $B>B_*$.

Let $\U$ be a proper state such that $B_* < B(u) < \bar{B}$.
We start with $D_0=D(u_0)>0$.
Consider the stopping time
$\T_1 = \T(u_0)$ defined by~(\ref{eq:stopping}) and define $D_1 = D(u_{\T_1})$.
For the shifted process $(\U_{\T_1+t})_{t\geqslant 0}$,
consider the stopping time $\T_2 = \T(u_{\T_1})$ and write $D_2 =
D(u_{\T_1+\T_2})$.
Analogously, once $\T_1,\T_2,\dots,\T_n$ have been constructed,
consider, for the shifted process
$(\U_{\T_1+\T_2+\cdots+\T_n+t})_{t\geqslant 0}$,
the stopping time
$\T_{n+1} = \T(u_{\T_1+\T_2+\cdots+\T_n})$,
and write
$D_{n+1} = D(u_{\T_1+\T_2+\cdots+\T_n+\T_{n+1}})$.
Let $\gamma = \inf\{n:D_n \leqslant 0\}$.

Taking $s=D_0$, it follows
from~(\ref{eq:BoundedDriftD})
that \[(D_{n\wedge \gamma})_{n=0,1,2,\dots} \preccurlyeq ([S_{n\wedge
\sigma}]^+)_{n=0,1,2,\dots},\] whence $\gamma \preccurlyeq \sigma$.
Therefore we get
\begin{align*}
\frac{\tau_{\{B\leqslant B_*\}}}{\Psi B_*}
\leqslant \sum_{n=1}^\gamma \frac{\T_n}{\Psi B_*}
\leqslant
\sum_{n=1}^\gamma e^{D_{n-1}}
\leqslant  \gamma  \exp\left[ \max_{0 \leqslant n < \gamma} D_{n}\right ]
\leqslant  \gamma e^{ D_0 + \Psi \gamma},
\end{align*}
whence by Lemma~\ref{lemmaHittingTimeTail}
\[
\frac{1}{\Psi B_*}{\EE^{\U}\left[\tau_{\{B\leqslant B_*\}}\right]}
\leqslant \EE^{\U} \left[\gamma e^{D_0 + \Psi \gamma}
\right]
\leqslant \EE^{s} \left[\sigma e^{s +\Psi \sigma}\right]
\leqslant \EE^{\bar{s}} \left[\sigma e^{\bar{s} + \Psi \sigma}\right] <
\infty,
\]
where $\bar{s}=\log \frac{\bar{B}}{B_*}$.
\end{proof}

\subsection{Downwards drift}
\label{sec:drift}

Write $A=A(u_0)$, $N=N(u_0)$, $B=B(u_0)$, $A'=A(u_\T)$, $N'=N(u_\T)$,
$B'=B(u_\T)$.
We decompose time in three parts:
\begin{equation}
\label{eq:totaltime}
\T = \mathscr{M} + \mathscr{S} + \mathscr{I}
,
\end{equation}
where
\[
\mathscr{M} = \int_0^{\T} \I_{\mbox{moving}}(s) \dd s
,
\quad
\mathscr{S} = \int_0^{\T} \I_{\mbox{serving}}(s) \dd s
,
\quad
\mathscr{I} = \int_0^{\T} \I_{\mbox{idle}}(s) \dd s.
\]
By definition of $\T_\circ$, the system cannot be idle for any $t<\T$, thus
$ \mathscr{I} = 0$.
For each $t>0$, let $\N_t$ denote the \emph{number of departures times} in
$(0,t]$.
Fix $\N=\N_\T$, the number of customers served up to time $\T$.
The total time spent with services during $(0,\T]$ is given by
$$\mathscr{S} = \sum_{n=1}^{\N} T_n + \beta T_{{\N}+1}$$
for some $0 \leqslant \beta < 1$, where $(T_n)_{n\in\NN}$
are i.i.d.\ exponential random variables.

Writing $A_t=A(u_t)$, it follows from~(\ref{eq:potential}) that $\frac{\dd
A_t}{\dd
t} = \lambda$ for Lebesgue-a.e.\ $t<\T$.
Moreover, $(A_t)_t$ jumps downwards at departure times,
and~(\ref{eq:explored}) reads as
\[
  A_t = \left[ \big. A_{t-} - E \right]^+.
\]
Since $A_t>0$ for all $t<\T$, $A'$ satisfies
\[
  A' = A + \lambda\T - \left( \textstyle \sum_{n=1}^{{\N}-1} E_n + \beta' E_{\N}
\right),
\]
where $0<\beta'\leqslant 1$ and
$(E_n)_{n\in\NN}$ are i.i.d.\ exponential random variables.

We now present the last ingredient, which is proved in the next subsection.
\begin{proposition}[Polling behavior]
\label{prop:polling}
The distribution of $\mathscr{M}$ under $\PP^{\,\U}$ is tight:
\[
\PP^{\,\U} \left\{ \big. \mathscr{M}>t \right\}
\mathop{\longrightarrow}_{t\to\oo} 0
\]
uniformly over all proper states $\U$.
\end{proposition}

\begin{proof}[Proof of Proposition~\ref{prop:drift}]
It follows from Donsker's invariance principle and
from Proposition~\ref{prop:polling}
that
\begin{align}
\label{eq:Tbehaveswell}
\textstyle \sum_{n=1}^{{\N}+1} T_n & < {\N} + \delta B,
\\
\label{eq:Ebehaveswell}
\textstyle \sum_{n=1}^{{\N}-1} E_n & > {\N} - \delta B,
\\
\label{eq:nozigzag}
\mathscr{M} & < \delta\Psi B,
\end{align}
hold with high probability as $B\to\infty$, uniformly in $\U$.

Assume that (\ref{eq:Tbehaveswell}), (\ref{eq:Ebehaveswell}), and
(\ref{eq:nozigzag}) happen.
Putting these altogether yields
\begin{align}
0 \leqslant A' & = A + \lambda \T - \textstyle \sum_{n=1}^{{\N}-1} E_n -
\beta' E_{\N}
\nonumber
\\
& \leqslant A + \lambda \T - {\N} + \delta B
&&
\mbox{by }
(\ref{eq:Ebehaveswell})
\nonumber
\\
& \leqslant A + \lambda \T - \mathscr{S} + 2\delta B
&&
\mbox{by }
(\ref{eq:Tbehaveswell})
\nonumber
\\
& = A + \lambda \T - \T + \mathscr{M} + 2\delta B
&&
\mbox{by }
(\ref{eq:totaltime})
\nonumber
\\
& \leqslant A -\eta \T + \delta\Psi B + 2\delta B
&&
\mbox{by }
(\ref{eq:nozigzag})
\nonumber
\\
& \leqslant A -\eta \T + 2\delta\Psi B
&&
\mbox{since }
\Psi>2
\label{eq:Adecreases}
\\
& <
2B - \eta \T = 
\eta(\T^+ - \T).
&&
\mbox{since }
A < B,\ 2\delta\Psi<1
\nonumber
\end{align}

By the last inequality, we have $\T < \T^+$.
It then follows from the definition of $\T$ that $U \subseteq G_{\T}$, whence
\(
u_{\T}(x) > - \T \mbox{ for } x \in U
.
\)
But by~(\ref{eq:uslow}) and the definition of $U$, we have
\(
u_{\T}(x) \geqslant u_0(x) - \T \geqslant -N/2 - \T \mbox{ for } x \in U^c
.
\)
Therefore,
\[
N' \leqslant N/2 + \T.
\]
Combining this and~(\ref{eq:Adecreases}):
\begin{align*}
B' - B & \leqslant (-\eta \T + 2 \delta \Psi B) + 4(\T -N/2)\epsilon
\\
& \leqslant - \left(2\epsilon - 2 \delta \Psi \right) B
- (\eta-4\epsilon)\T
&&
\mbox{since }
N>B
\\
& \leqslant - \left(2\epsilon - 2 \delta \Psi \right) B
&&
\mbox{since }
4\epsilon < \eta
\\
& = - {\epsilon} B.
&&
\qedhere
\end{align*}
\end{proof}

\subsection{Polling behavior}
\label{sec:polling}

In this section we prove Proposition~\ref{prop:polling} via a coupling with the
greedy server on the real line.
The latter eventually moves towards one of the two directions and spends little
time going backwards, which was shown in~\cite{FossRollaSidoravicius15}.
We consider a periodic extension of the initial potential $u$ on the circle,
and approximate it by another potential with less oscillations, for which we can
generalize this result.

The rules described in Section~\ref{sec:potential} may also be used to construct evolutions on the real line, which we will couple with the greedy server on the circle.

Let us start with an informal description of how these systems are coupled.

First, define an evolution on~$\RR$ by a lifting from~$\RR/\ZZ$: extend the potential periodically and replace both the server and the currently served customer by infinitely many replicas at unit interdistances.
If all the replicas evolve using the same randomness, this system remains periodic for all times.
Moreover, one can recover the original system by projecting back from the line back onto the circle, so it is basically the same system.

Now remove all the replicas and consider the system with a single server.
For short times, this server evolves just like in the system with all the replicas. In fact, this remains true until the first time when the server needs to query for the presence of customers in a region that has already been queried by another replica.
So the systems will uncouple at time~$\T_{[1]}$ defined below, or~$\T_\circ$ for the system on~$\RR/\ZZ$.

Finally, this system with a single server on~$\RR$ can be coupled with another similar system also on~$\RR$, which at $t=0$ starts with the same positions for both the server and served customer, but a slightly different potential.
If the same randomness is used for both systems, the servers will evolve together until the first time when they query for the presence of customers in a region where the potential was initially different.
This time of uncoupling is given by~$\T_{U}$ defined below.
It typically occurs before~$\T_{[1]}$, in which case it will correspond to~$\T_\gamma$ on the circle.

In summary, we can couple the system on the circle with one on the line, and the latter with another one having different initial potential, this double coupling lasting until $\T_{U}\wedge T_{[1]}$.
Figure~\ref{fig:coupling} shows an example where $\T_{U}$ is attained first, and Figure~\ref{fig:stopping}(b) shows an example on the circle where $\T_{[1]}$ would be attained first.

We now make the above description precise.

\paragraph{Coupling with the greedy server on~$\RR$}

A \emph{potential} is a piecewise continuous, upper semi-continuous function
$\bar{u}(x)\leqslant 0$ on~$\RR$ with $\int_\RR -\bar{u}\dd x = \oo$.
The evolution of $(\bar{\U}_t)_{t\geqslant 0}$ is
defined in the same way as
on the circle, i.e.,
satisfying~(\ref{eq:movement}),(\ref{eq:potential}),(\ref{eq:newcustomer}),(\ref
{eq:explored}).

Let $\U$ be a proper state on the circle and $\bar{u}$ the
\emph{periodic extension} of $u$ on $\RR$.
Without loss of generality, in the sequel we assume that $\S=0$.
Take $\bar{\S}=0$ and let $\bar{\C}$ be the only representant of $\C$ in
$[-\frac{1}{2},\frac{1}{2})$.

We define
\begin{equation*}
H_t  = \left\{ \big. x \in\RR: \bar{u}_t(x) > -t \right\}
.
\end{equation*}
By the same arguments as for the greedy server
on $\RR/\ZZ$, $H_t$ is nondecreasing in $t$;
it is empty until the first departure time, after which
it consists of a closed interval containing both $\bar{\S}_t$ and $\bar{\C}_t$.

Define the stopping time
\[
\T_{[1]} = \inf\left\{ t: |H_t| \geqslant 1 \right\}.
\]
For each $t<\T_{[1]}$, we define the map $\pi_t$ that takes each
point $x\in[L(t),L(t)+1)\subseteq \RR$ to its projection $x\in\RR/\ZZ$, where
$L(t)$ is chosen as follows.
If $H_t=\emptyset$, we take $L(t)=-\frac{1}{2}$;
otherwise if $H_t \ne \emptyset$, we take $L(t)=\inf H_t$.
For a function $w:\RR\to\RR$ define $\pi_t w=w\circ \pi_t^{-1}$.

Recall from~(\ref{eq:valley}) that the set $U$ is either the whole circle or an
open arc not containing $\S=0$.
Let
\[
\bar{U}=\left\{ x\in\RR : \bar{u}(x) < \textstyle - \frac{N(u)}{2} \right\}
\]
and take
\[
l=\inf \left( \big. \bar{U}\cap[-1,0] \right)
, \quad
r=\sup \left( \big. \bar{U}\cap[0,1] \right).
\]

Finally, consider another initial state $\tilde{\U}$ given by $\tilde{\S}=\bar{\S}$, $\tilde{\C}=\bar{\C}$, and
\begin{equation}
\label{eq:utilde}
\tilde{u}(x) =
\begin{cases}
\bar{u}(x), & x\in(l,r), \\
-\frac{N}{2}, & \mbox{otherwise}.
\end{cases}
\end{equation}
Define the evolution $(\tilde{\U}_t)_{t\geqslant 0}$
again by the same rules as for $\bar{\U}$,
and consider the stopping time
\[
\T_{U} = \inf\left\{ \big. t: H_t \not\subseteq (l,r)
\right\}.
\]

\begin{lemma}
[Coupling]
\label{lemma:coupling}
The evolutions $(\bar{\U}_t)_{t\geqslant 0}$ and
$(\tilde{\U}_t)_{t\geqslant 0}$ on the line
and $(\U_t)_{t\geqslant 0}$ on the circle
may be constructed on the same probability space, satisfying
\begin{align*}
\T_\circ &= \T_{[1]},
&
\T_\curlyvee &= \T_{U}\wedge T_{[1]},
\\
\U_t &=
\pi_t\big( \bar{\U}_t \big) \mbox{ for all } t<\T_{[1]},
&
\bar{\S}_t &=
\tilde{\S}_t  \mbox{ for all } t<\T_{U}.
\end{align*}
\end{lemma}

\begin{proof}
The coupling given by Lemma~\ref{lemma:coupling} is illustrated in
Figure~\ref{fig:coupling}.
\begin{figure}[b!]
\begin{center}
\includegraphics[width=.95\textwidth]{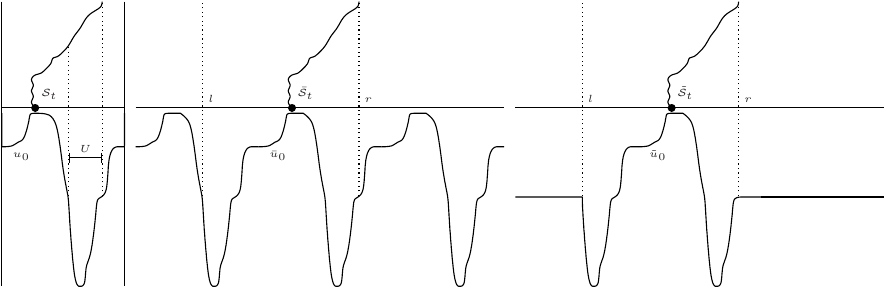}
\caption{Illustration of how $\U$, $\bar{\U}$, and $\tilde{\U}$ evolve
together until time $\T_\curlyvee$.}
\label{fig:coupling}
\end{center}
\end{figure}
The evolution of $(\U_t)_t$ can be constructed using an i.i.d.\
sequence $(E_n,U_n,T_n)_n$, where $E_n$ and $U_n$ are the exponential uniform
used as input for~(\ref{eq:newcustomer}) and~(\ref{eq:explored}) at each
departure time $t_n$, and $T_n$ are the service times prior
to the $n$-th departure time.
(When the system enters the idle state, another clock will be needed to
determine the next arrival time, but this state cannot be achieved before
$\T_\circ$.)

The coupling is simple: we use the same sequence $(E_n,U_n,T_n)_n$ to build
$(\bar{\U}_t)_t$ and
$(\tilde{\U}_t)_t$.
It remains to check that this coupling a.s.\ satisfies the identities
stated in the lemma; the details are left to the reader.
\end{proof}

\paragraph{Strong transience}

For the evolution $(\tilde{\U}_t)_t$,
the total distance traveled by the server between times $t$ and $t'$ is denoted
by
\begin{equation}
\nonumber
V_t^{t'}(\tilde{\S}_\cdot) := \int_t^{t'} |\tilde{\V}_s|\dd s = v \int_t^{t'}
\I_{\mbox{moving}}(s)
\dd s
.
\end{equation}

We say that $(\tilde{\S}_t)_{t \geqslant 0}$ is \emph{transient} if,
for each
$M>0$, $\sup\big\{ t:\tilde{\S}_t \in [-M,M] \big\}<\infty$.
If moreover
\begin{equation}
\label{eq:walkdisplace}
\left| \tilde{\S}_t-\tilde{\S}_0 \right| \geqslant
\frac{1}{3}V_0^t(\tilde{\S}_\cdot)
\qquad
\mbox{for all } t>0
,
\end{equation}
we say that $(\tilde{\S}_t)_{t \geqslant 0}$ is \emph{strongly
transient}.
The latter means that the total displacement must increase linearly with the
traveled distance.

For $0 \leqslant \alpha \leqslant 1$, we say that $\tilde{\U}$
is \emph{$\alpha$-unimodal} if $\tilde{u}$ attains its maximum on $\tilde{\C}$, and
\begin{equation}
\label{eq:alphaunimodal}
\tilde{u}(x) \leqslant \alpha \cdot  \inf_{y\in[\tilde{\C},x]}
\big. \tilde{u}(y)
,
\
\forall\ x > \tilde{\C}
,
\qquad
\tilde{u}(x) \leqslant \alpha \cdot  \inf_{y\in[x,\tilde{\C}]}
\big. \tilde{u}(y)
,
\
\forall\ x < \tilde{\C}
.
\end{equation}
Notice that with $\alpha=1$ this is equivalent to $\tilde{u}$
being nondecreasing on $(-\oo,\tilde{\C}]$ and
nonincreasing on $[\tilde{\C},\oo)$, and for $\alpha<1$ and $\tilde{u}\leqslant 0$, the condition is weaker.

The result below is a consequence of Proposition~1
in~\cite{FossRollaSidoravicius15}, written in our notations.

\begin{proposition}
Given any $\tilde{\U}$ that is $\alpha$-unimodal with
${\alpha=1}$,
$(\tilde{\S}_{t})_{t\geqslant
0}$ is a.s.\ transient.
\end{proposition}

In order to prove Proposition~\ref{prop:polling},
we shall obtain a slightly stronger result:
\begin{proposition}
\label{prop:transient}
Given any $\tilde{\U}$ that is $\alpha$-unimodal with
$\alpha=\frac{1}{2}$, there exists a random time $\T_Z$ satisfying
$\PP^{\,\tilde{\U}} (\T _Z < \infty) = 1$, and such that
$(\tilde{\S}_{\T_Z+t})_{t\geqslant
0}$ is strongly transient.
Moreover, the number of departure times $\tilde{\N}_{\T_Z}$ before $\T_Z$
is tight:
\[
\PP^{\,\tilde{\U}} \left\{ \big. \tilde{\N}_{\T_Z} > k
\right\} \mathop{\longrightarrow}_{k\to\oo} 0
\]
uniformly over all $\alpha$-unimodal $\tilde{\U}$.
\end{proposition}

Proposition~\ref{prop:transient} is proved in
Sections~\ref{sec:block}~and~\ref{sec:renewal}
by
adapting
the multi-scale construction of~\cite{FossRollaSidoravicius15} to the case of
$\alpha$-unimodal initial states.

\begin{proof}
[Proof of Proposition~\ref{prop:polling}]
We first observe that $\tilde{\U}$, with $\tilde{u}$ defined
by~(\ref{eq:utilde}), is
$\alpha$-unimodal for $\alpha=\frac{1}{2}$.
By definition of $\mathscr{M}$ and $V$,
\[
\mathscr{M} = \frac{1}{v}\ V_0^\T(\S_\cdot)
\leqslant
\frac{1}{v}\ V_0^{\T_\curlyvee}(\S_\cdot)
\]
and by Lemma~\ref{lemma:coupling},
\[
V_0^{\T_\curlyvee}(\S_\cdot)
= V_0^{\T_{[1]}\wedge \T_{U}}(\bar{\S}_\cdot)
= V_0^{\T_{[1]}\wedge \T_{U}}(\tilde{\S}_\cdot)
\leqslant V_0^{\T_{U}}(\tilde{\S}_\cdot)
.
\]

By definition of $\T_U$, we have that $\tilde{\S}_t \in [l,r] \subseteq
[-1,1]$ for all $t<\T_U$.
The distance traveled by the server between
consecutive departure times is thus bounded by $2$, and therefore
\[
V_0^{\T_{Z} \wedge \T_U}(\tilde{\S}_\cdot)
\leqslant
2 (\tilde{\N}_{\T_Z \wedge \T_U}+1)
\leqslant
2 (\tilde{\N}_{\T_Z}+1)
.
\]
In case $\T_U \leqslant \T_Z$, this upper bound for
$V_0^{\T_{U}}(\tilde{\S}_\cdot)$ is good enough.
So consider the case $\T_Z<\T_U$ and write
\[
V_0^{\T_{U}}(\tilde{\S}_\cdot)
=
V_0^{\T_{Z}}(\tilde{\S}_\cdot)
+
V_{\T_{Z}}^{\T_{U}}(\tilde{\S}_\cdot).
\]
By~(\ref{eq:walkdisplace}) and the definition of $\T_Z$,
\[
V_{\T_{Z}}^{\T_{U}}(\tilde{\S}_\cdot)
\leqslant
6
.
\]
Summarizing,
\[
\mathscr{M} \leqslant \frac{1}{v} \left( 8 + 2 \tilde{\N}_{\T_Z} \right)
\]
and the result then follows from Proposition~\ref{prop:transient}.
\end{proof}

\subsection{Multi-scale estimates}
\label{sec:block}

In the remainder of this section we give a short but self-contained proof of
Proposition~\ref{prop:transient} using a block argument.
The reader will find a similar construction, with a
more extended explanation of the main ideas,
in~\cite{FossRollaSidoravicius15}.
Since only $\tilde{\U}$ is concerned, we shorten notation and write $\U$
instead.
Each time $C$ or $c$ appears, it denotes a different constant that is positive,
finite, and depends only on $v$.

\medskip

Let $\A_t=\sigma( (\U_s)_{s\in[0,t]} )$ denote the natural filtration for
$(\U_t)_{t\geqslant 0}$.
We construct a sequence of stopping times $0=L_0<L_1<\cdots$
and define the corresponding events of success $A_j \in \A_{L_{j+1}}$ in
terms of $\U_{L_j}$.
The construction will have the following properties.
For some sequence $p_j$ and any $\U$ that is $\alpha$-unimodal,
\begin{equation}
\label{eq:blockestimateprobab}
\PP^{\,\U} ( A_{j} | \A_{L_j} )
=
\PP^{\,\U} ( A_{j} | \U_{L_j} )
\geqslant p_j
\mbox{ on }
A_0 \cap \cdots \cap A_{j-1},
\quad
\mbox{and}
\quad
\prod_j p_j > 0
.
\end{equation}
The event $\cap_{j=0}^\infty A_j$ implies strong transience of
$({\S}_t)_{t\geqslant 0}$.
Almost surely, for each $j=0,1,2,\dots$, the state of $\U_{L_j}$ is
\emph{serving}.

\medskip

We assume without loss of generality that the state of $\U_0$ is
serving, that $\lambda=1$, and that $\S_0=0$.
Take $\sigma = \sgn \S_{L_1}$ to indicate the direction in which subsequent
blocks are supposed to grow.
Let
$Z_j=\sigma \S_{L_j}$,
$N_j = L_j-u(\S_{L_j})$,
$Q_j = \N_{L_{j+1}}-\N_{L_j}$,
$X_{j} = Z_{j+1} - Z_{j}$,
$M_j=L_{j+1}-L_j$.

The triggering step $j=0$ is defined as follows.
We always take $Q_0=1$, and the first step consists of finishing with
the customer that is being served at time $L_{0}$, then traveling towards the
nearest customer at position $\sigma Z_1$, and $L_1$ is the stopping time
attained as soon as the server reaches this position.
The event $A_0$ means \emph{success} at the Step~$j=0$, and is defined by the
following conditions:
$X_0^- \leqslant X_0 \leqslant X_0^+$
and
$M_0^-  \leqslant M_0 \leqslant M_0^+$,
where
$
X_0^- = \frac{9}{N_1},
$
$
X_0^+ = 36,
$
$
M_0^- = 1,
$
and
$
M_0^+ = 2 + \frac{36}{v}.
$
If there is no success, we declare Step~$0$ to have \emph{failed} and stop.
In the sequel we assume without loss of generality that $\sigma=+1$.

For $j\geqslant 1$, suppose that Steps~$0,1,2,\dots,j-1$ have been successful
and start from $u_{L_j}$, and take $\ell_j = \lceil 54j^{1/4} \rceil$, $D_j =
\frac{1}{36}\ell_j$.
Let $s_{j,0}=S_{L_j}$ and $s_{j,1},s_{j,2},\dots,s_{j,\ell_j},s_{j,\ell_j+1}$
denote the positions of the next $\ell_j+1$ customers.
Step~$j$ may be successful, which is denoted by the event $A_j$, in two
situations:
first, if $s_{j,n}>s_{j,n-1}$ for $n=1,\dots,\ell_j$,
in which case we take $Q_j=\ell_j$;
second, if there is only one $\tilde n\in \{1,\dots,\ell_j+1\}$ such that
$\tilde{n} \ne \ell_j+1$ and $s_{j,\tilde{n}}<s_{j,\tilde{n}-1}$,
in which case we take $Q_j=\ell_j+1$.
If none of these two happen, we declare Step~$j$ to have \emph{failed} and stop.
Otherwise, in either of the above two cases we say that Step~$j$ is
\emph{successful} if moreover
\begin{align}
\nonumber
\begin{cases}
Q_j^- \leqslant Q_j \leqslant Q_j^+
,
\quad
X_j^- \leqslant X_j \leqslant X_j^+
,
\quad
M_j^- \leqslant M_j \leqslant M_j^+
,
\\
V_{L_j}^{L_{j+1}}(\S_\cdot) \leqslant X_j + \frac{4 D_j}{N_j}
,
\quad
Z_{j-1} < \S_t < Z_{j+1}
\mbox{ for }
L_j \leqslant t < L_{j+1}
,
\end{cases}
\end{align}
where
$Q_j^- = \ell_j,
Q_j^+ = \ell_j+1,
X_j^- = \frac{\ell_j-1}{3N_{j+1}},
X_j^+ = \frac{3 \ell_j}{N_j},
M_j^- = \frac{1}{2} Q_j^-,
M_j^+ = 2 Q_j^+ + \frac{3 X_j^+}{v}.
$
Here the time $L_{j+1}$ is given by the instant when the server reaches the
last customer, located at $Z_{j+1}$, and the next block starts with this
customer being served.

\medskip

We now estimate the probability of success $\PP^{\,\U}(A_j|\U_{L_j})$ on
$A_0 \cap \cdots \cap A_{j-1}$
by considering a number of events that imply $A_j$.

First notice that $A_0 \cap \cdots \cap A_{j-1}$ implies that
$
M_j^- \geqslant C j^{1/4},
N_j = -u(\sigma Z_j)+L_j \geqslant L_j \geqslant M_0^-+\cdots+M_{j-1}^-
\geqslant C j^{5/4},
X_j^+ \leqslant C j^{-1},
X_{j-1}^- \leqslant C j^{-1},
$
and thus
$
M_j^+ \leqslant C j^{1/4}.
$
Moreover, the condition of~$\U_t$ being $\alpha$-unimodal is preserved for all~$t$.
(Indeed, while the system state is serving or moving, by~(\ref{eq:potential}) the potential $u \leqslant 0$ changes by the subtraction of a constant that increases with time, and this preserves~(\ref{eq:alphaunimodal}). At departure times, the state is updated by~(\ref{eq:newcustomer})-(\ref{eq:explored}), which changes the position of $\C$ and increases the value of~$u$ to~$0$ on an interval that contains both the old and new~$\C$. This increases the $\inf$ in~(\ref{eq:alphaunimodal}), so this inequality still holds for $x$ outside of such interval, whereas for $x$ in such interval both $u(x)$ and the $\inf$ become $0$.)
Thus, for $j \geqslant 1$, the event $A_{j-1}$ implies the following conditions on
$\U_{L_j}$:
\begin{align}
\label{eq:plateau}
\begin{cases}
u_{L_{j}}(x) = u_0(x)-L_j \leqslant - \frac{N_j}{2}
&
\mbox{for }
x > Z_{j},
\\
u_{L_{j}}(x) \geqslant - M_{j-1}
&
\mbox{for }
Z_{j} - X_{j-1}^- < x < Z_{j}.
\end{cases}
\end{align}

Let $T_0,T_1,\dots,T_{\ell_j}$ denote the service times of the customer being
served at time $L_0$ and the following $\ell_j$ customers.
Let $E_1,U_1,E_2,U_2,\dots,E_{\ell_j},U_{\ell_j},E_{\ell_j+1},U_{\ell_j+1}$ be
the exponential and uniform random variables used for determining the positions
of $s_{j,1},s_{j,\ell_j+1}$ via~(\ref{eq:newcustomer}).

For $j=0$, consider the event that $1<T_0<2$, $36<E_1<72$, and that $U_1$ lies
on the largest interval among $(0 ,\frac{a}{a+b})$ and $[\frac{a}{a+b},1)$.
The probability that these conditions are satisfied is at least
$p_0 = \frac{1}{2}(e^{-36}-e^{-72})(e^{-1}-e^{-2}) > 0$.
The requirement for $U_1$ implies that $u_0(S_{L_1}) \leqslant u_0(-S_{L_1})$.
Hence, by $\alpha$-unimodality of $u_0$, the occurrence of
the above events imply
\[
36
<
\int_{-z}^{+z} [-u_0(x)+T_0] \dd x
\leqslant
\int_{-z}^{+z} \max_{[{-z},{+z}]}(T_0 -u_0) \dd x
\leqslant
- 4 X_0 [u_0(S_{L_1}) + T_0]
\leqslant
4 X_0 N_1
\]
and
\[
72
>
\int_{-z}^{+z} [-u_0(x)+T_0] \dd x
\geqslant
\int_{-z}^{+z} T_0 \dd x
\geqslant
2 z
,
\]
which in turn imply the bounds on $X_0$ and $M_0$,
therefore $\PP^u(A_0) \geqslant p_0 > 0$.

For $j\geqslant 1$, we observe that the positions $s_{j,1},s_{j,2},\dots
s_{j,\ell_j+1}$ can be sampled via a Poisson Point Process on
the region
$$R=\left\{ (x, t)\in \RR^2 : Z_j - X_{j-1}^- \leqslant x \leqslant Z_j
+ X_j^+, u_{L_j}(x) < t \leqslant M_j^+ \right\},$$ unless the elapsed time $M_j$
exceeds $M_j^+$ or the exploration for customers leaves the interval
$[Z_j-X_{j-1}^-,Z_j+X_j^+]$, which is ruled out a posteriori in case Step~$j$ is
successful.
This region $R$ can be decomposed in a disjoint union $R_1 \cup R_2$, where
$R_1 = \left\{ (x, t)\in \RR^2 : Z_j \leqslant x \leqslant Z_j + X_j^+,
u_{L_j}(x) < t \leqslant 0 \right\}$
and
$R_2 \subseteq [Z_j-X_{j-1}^-,Z_j] \times [-M_{j-1}^+,M_j^+]
\cup [Z_j,Z_j+X_j^+] \times [0,M_j^+]$.

The inequalities in~(\ref{eq:plateau}) and those preceding it imply that
$$|R_2| \leqslant (X_{j-1}^- + X_j^+)(M_{j-1}^+ + M_{j}^+) \leqslant C j^{-3/4}.$$
The probability that there are two or more points in $R_2$ is thus bounded by
$Cj^{-3/2}$.

Define $\beta(x) = \int_{Z_j}^{x} [-u_{L_{j}}(z)] \dd z$, $x \geqslant Z_j$,
and write $(x_1,t_1),(x_2,t_2),\dots$ the set of points found in $R_1$,
labeled by ordering $x_1<x_2<x_3<\cdots$.
Writing $x_0 = Z_j$, we have that $\left( \beta(x_n)-\beta(x_{n-1})
\right)_{n=1,\dots,\ell_j}$ are i.i.d. exponential random variables.
With positive probability, and tending to $1$ exponentially fast in $\ell_j$ as
$j\to\infty$, both events
\[
\frac{2}{3}(\ell_j-1) \leqslant \beta(x_{\ell_j-1}) \leqslant \beta(x_{\ell_j})
\leqslant \frac{3}{2}\ell_j
\]
and
\[
\beta(x_n)-\beta(x_{n-1}) \leqslant D_j \quad \mbox{ for } n=1,2,\dots,\ell_j
\]
occur.
But since $\U_{L_j}$ is $\alpha$-unimodal we have
\[
\frac{1}{2}N_j
\leqslant
\frac{\beta(x_{n})-\beta(x_{n-1})}{x_n - x_{n-1}}
\leqslant
2 N_{j+1}
.
\]
The above inequalities imply that
\(
X_j^- \leqslant x_{\ell_j-1}-Z_j \leqslant x_{\ell_j}-Z_j \leqslant X_j^+
\)
and
\(
x_{n}-x_{n-1} \leqslant \frac{2 D_j}{N_j} = \frac{\ell_j}{18 N_j} \leqslant
\frac{X_{j-1}^-}{3}
.
\)
On the event that there are no points in $R_2$, we have
$V_{L_j}^{L_{j+1}} = X_j = x_{\ell_j} - Z_j$.
On the event that there is one point in $R_2$, we have
$x_{\ell_j-1} - Z_j \leqslant X_j \leqslant x_{\ell_j} - Z_j$ and
$V_{L_j}^{L_{j+1}} \leqslant X_j + 4\frac{D_j}{N_j}$.
This finishes the bounds on $Q_j$, $X_j$,
$V_{L_j}^{L_{j+1}}$, and $(\S_t)_{t\in [L_j,L_{j+1})}$.

It remains to control $M_j$, which
is given by the sum $T_{0}+\cdots+T_{Q_j-1} + v^{-1}
V_{L_j}^{L{j+1}}(\S_\cdot)$ of service times plus traveling time.
The latter is non-negative and bounded by $2X_j/v$, which is bounded by
$2X_j^+/v$.
Therefore the inequality $M_j^- \leqslant M_j \leqslant M_j^+$ holds whenever
$Q_j/2 < T_{0}+\cdots+T_{Q_j-1} < 2Q_j$,
which in turn occur with exponentially high probability in $\ell_j$.

\medskip

Finally, we use the bound on $V_{L_j}^{L_{j+1}}(\S_\cdot)$ to prove strong
transience.
We add the requirement that $V_{L_j}^{L_{j+1}}(\S_\cdot)=X_j$ for $j=1$.
This changes the lower bound on probability of $A_1$, but it remains positive.
Notice that the same equality is true for $j=0$ by construction.
Now one can decompose $V_0^t(\S_\cdot)$ in distances traveled in each of the two
possible directions and again decompose these distances in the contribution from
each block, and combine the bounds on $X_{j-1}$ with
$\frac{4 D_j}{N_j} \leqslant \frac{2}{3} X_{j-1}^-$ to get
$V_0^t(\S_\cdot) \leqslant \frac{5}{3} {\S}_t$.

\subsection{Renewal argument}
\label{sec:renewal}

Having~(\ref{eq:blockestimateprobab}) in hands, we finally prove
tightness of $\N_{\T_Z}$.
We need that the probability
of success in each block~$j$ not only be bounded from below by some $p_j$ but
actually equal to $p_j$.
We introduce an artificial coin toss to provide this last ingredient.

Enlarge the underlying probability space to add an independent sequence of
i.i.d.\ uniform variables $\tilde U_{j}$.
For each $j$, define the event $\tilde{A}_j \subseteq A_j$ by
\[
\tilde{A}_j = A_j \cap \left[  \tilde{U}_{j+1} < \frac{p_j}{\PP(A_j | \U_{L_j})}
\right]
.
\]
In words, we add an extra coin toss in order to have an exact equality instead
of an upper bound:
\[
\PP^{\,\U} ( \tilde{A}_j | \U_{L_j} )
= p_j
\mbox{ on }
\tilde{A}_0 \cap \dots \cap \tilde{A}_{j-1}.
\]

Notice that $J^0 = \min \{ j: \tilde{A}_{j-1} \mbox{ does not occur} \} \in
\{1,2,3,\dots\}\cup\{\oo\}$ is a stopping time with respect to
$\{ \tilde{\A}_{j} \}_{j=0,1,2,\dots}$, where $\tilde{\A}_j =
\sigma(\A_{L_j}, \tilde{U}_0,\tilde{U}_1,\dots,\tilde{U}_{j-1},\tilde{U}_{j})$.
The distribution of $J^0$ is given by
$\PP(J^0>k+1)=p_0 p_1 \cdots p_k$.
If $J^0 = \oo$, we have success for all $j$ and $(\S_t)_{t\geqslant 0}$ is
strongly transient, and we can take $\T_Z=0$.

If otherwise, $J^0<\infty$, we have $\N_{L_{J^0}} \leqslant Q^+(J^0)$, where
$Q^+(j) = Q_0^+ + Q_1^+ + \cdots + Q_{j}^+$.
In this case we can apply a time shift of $L_{J^0}$ and define
$(\U^1_t)_{t \geqslant 0}$ by
$\U^1_t=\U_{t+L_{J^0}}$.
For this evolution $(\U^1_t)_{t \geqslant 0}$ we can define the
stopping times $L^1_0,L^1_1,L^1_2,\dots$, the events
$\tilde{A}^1_0,\tilde{A}^1_1,\tilde{A}^1_2,\dots$, and the step of first
failure $J^1$.

By the strong Markov property, the conditional distribution of
$(\U^1_t)_{t \geqslant 0}$ given that $J^0 < \oo$ is given by
$\PP^{\,\U_{L_{J^0}}}$, and since
$\U_{L_{J^0}}$ is $\alpha$-unimodal, the
conditional distribution of $J^1$ given that $J^0<\oo$ is the same:
$\PP(J^1>k+1|J^0<\oo)=p_0 p_1 \cdots p_k$.

Again, if $J^1=\oo$, $(\S_{t+L_{J^0}})_{t\geqslant 0}$ is strongly transient
and we take $\T_Z = L_{J^0}$.
Otherwise, $\N_{L_{J^0}+L^1_{J^1}} \leqslant Q^+(J^0) +  Q^+(J^1)$.
Analogously we can construct $\U^2,J^2,\U^3,J^3,\dots$ until at some step
$K+1$ we get $J^{K+1} = \oo$.
We then take $\T_Z = L_{J^0} + L^1_{J^1} + \cdots + L^{K}_{J^K}$,
and we have that $(\S_{t+\T_Z})_{t \geqslant 0}$ is strongly transient.
As before, $\N_{\T_Z} \leqslant Q^+(J^0) +
Q^+(J^1) + \cdots + Q^+(J^K)$.
But the distribution of the latter upper bound does not depend on~$\U$,
and therefore $\N_{\T_Z}$ is tight.

\section*{Acknowledgments}

We are grateful to S.~Foss, who introduced us to this problem.
We thank M.~Jara for useful discussions.
This project had supported from grants
PICT-2015-3154,
PICT-2013-2137,
PICT-2012-2744,
PIP 11220130100521CO,
Conicet-45955,
MinCyT-BR-13/14 and
MathAmSud-2014-LSBS.

\renewcommand{\baselinestretch}{1}
\parskip 0pt
\small
\bibliographystyle{bib/rollaalphasiam}
\bibliography{bib/leo}

\end{document}